\documentclass[12pt]{article}
\usepackage{amsmath,amssymb,amsthm}
\usepackage{amsfonts}
\usepackage{graphicx}
 
\newtheorem{Theorem}{Theorem}[section]
\newtheorem{Lemma}[Theorem]{Lemma} 
\newtheorem{Proposition}[Theorem]{Proposition}
\newtheorem{Corollary}[Theorem]{Corollary}
\theoremstyle{definition}
\newtheorem{Definition}[Theorem]{Definition}

\def\C{\mathbb{C}} 
\def\ei{\mathbf{i}}
\def\ej{\mathbf{j}}
\def\ek{\mathbf{k}}
\def\D{\mathcal{D}}

\def\H{\mathbb{H}}
\def\M{\mathcal{M}} 
\def\R{\mathbb{R}} 
 
\def\L{\mathcal{L}} 
\def\S{\mathcal{S}} 
\def\Mat{M}
\DeclareMathOperator{\Ker}{Ker}

\DeclareMathOperator{\ce}{ce}
\DeclareMathOperator{\se}{se}
\DeclareMathOperator{\Ce}{Ce}
\DeclareMathOperator{\Se}{Se}
\DeclareMathOperator{\Sc}{Sc} 
\DeclareMathOperator{\VEC}{Vec} 
\DeclareMathOperator{\re}{Re} 
\DeclareMathOperator{\im}{Im} 
 
\newcommand{\pic}[4]{\vspace{1ex}\setlength{\unitlength}{.05\textwidth}
\begin{picture}(0,0)(0,0)
\put(#2){\includegraphics[#4]{#1.jpg}} 
\end{picture}\vspace*{#3cm}\vspace{1ex}} 

\begin{document}

\begin{center}
 {\Large Reduced-quaternionic Mathieu functions, time-dependent Moisil-Teodorescu operators, and the imaginary-time wave equation}

\medskip
\medskip

Jo\~{a}o Morais\\
R. Michael Porter

 \end{center}

\begin{abstract}
We construct a one-parameter family of generalized Mathieu
  functions, which are reduced quaternion-valued functions of a pair
  of real variables lying in an ellipse, and which we call
  $\lambda$-reduced quaternionic Mathieu functions.  We prove that the
  $\lambda$-RQM functions, which are in the kernel of the
  Moisil-Teodorescu operator $D+\lambda$ ($D$ is the Dirac operator
  and $\lambda\in\R\setminus\{0\}$), form a complete orthogonal system
  in the Hilbert space of square-integrable $\lambda$-metamonogenic
  functions with respect to the $L^2$-norm over confocal
  ellipses. Further, we introduce the zero-boundary
  $\lambda$-RQM-functions, which are $\lambda$-RQM functions whose
  scalar part vanishes on the boundary of the ellipse. The limiting
  values of the $\lambda$-RQM functions as the eccentricity of the
  ellipse tends to zero are expressed in terms of Bessel functions of
  the first kind and form a complete orthogonal system for
  $\lambda$-metamonogenic functions with respect to the $L^2$-norm on
  the unit disk.  A connection between the $\lambda$-RQM functions and
  the time-dependent solutions of the imaginary-time wave equation in
  the elliptical coordinate system is shown.
\end{abstract}


\medskip



\section{Introduction}

Hypercomplex analysis is concerned with the study of elliptic partial
differential operators of the form
$\sum_{i=0}^m e_i\,\partial/\partial x_i$ defined on functions taking
values in a Clifford algebra over the Euclidean space $\R^{m+1}$ with
ideal units $e_i$. W.\ Hamilton introduced such an operator in the
context of quaternions in \cite{WRH1853} and it was studied
extensively by R.\ Fueter \cite{Fueter1940,Fueter1949}. Functions that
are annihilated by this operator are called monogenic (or regular, or
holomorphic, or hyperholomorphic).  One may study a large variety of
versions of $\R^{m+1}$-valued (paravector) or $\R^{m}$-valued (vector)
monogenic functions defined in domains in $\R^d$ of diverse dimensions
$d$
\cite{Abul-EzConstales1990,BrackxDelangheSommen1982,CacaoFalcaoMalonek2011,Delanghe2007,Delanghe2009,GuerlebeckHabethaSproessig2008,GuerlebeckHabethaSproessig2016,MoraisHabilitation2019};
each case has its particularities.

Let us consider $d=2$ and the first-order partial differential quaternionic operator
\begin{equation*}
D = \ei\frac{\partial}{\partial x} + \ej\frac{\partial}{\partial y}
\end{equation*}
in planar domains. The Euclidean space $\R^3$ can be naturally
embedded in the quaternions to form ``reduced quaternions'' in the
real linear span of $1$, $\ei$, and $\ej$
\cite{MoraisHabilitation2019}. In \cite{LuPeRoSha2013}, a detailed
investigation of the Moisil-Teodorescu operator $D+\lambda$ ($\lambda\in\R-\{0\}$) was
carried out for reduced-quaternion-valued functions defined in
elliptical domains. The natural representation of $D$ in elliptical
coordinates gave rise to an elliptical version of the Cauchy kernel
for $D+\lambda$ ($\lambda$-metamonogenic functions) and related
fundamental integral formulas. It was pointed out that one may obtain
$\lambda$-metamonogenic functions of two variables as specific
products of pairs of Mathieu functions of the elliptical coordinates.
As was shown in \cite{LuPeRoSha2013}, there is an intimate relation
between $D+\lambda$ and the Helmholtz operator
$\Delta+\lambda^2$. Applications of the theory of metamonogenic
functions can be found in \cite{Morais2014} and
\cite{MoraisRosaKou2015}.

Here we develop this idea further. In Section
\ref{Quaternionic_Mathieu_functions}, we give a complete orthogonal
system of reduced-quaternion $\lambda$-metamonogenic functions in an
elliptical domain. The basis elements are of the particular form
described in \cite{LuPeRoSha2013}, but special attention must be given
to the underlying parameters. We call them $\lambda$-reduced
quaternionic Mathieu functions ($\lambda$-RQM functions) and denote
them $\Mat^{\pm}_n[\lambda]$. Section
\ref{Zero-boundary_RQM-functions} introduces the zero-boundary
$\lambda$-RQM-functions, which are $\lambda$-RQM functions whose real part vanishes
on the boundary of the ellipse of given eccentricity $\mu$. We prove
that they form a complete orthogonal system in the Hilbert space of
$\lambda$-metamonogenic functions with respect to the $L^2$-norm over
ellipses. The limiting values of the $\lambda$-RQM functions as
$\mu \to 0$ are expressed in terms of Bessel functions of the first
kind and form a complete orthogonal system for $\lambda$-metamonogenic
functions with respect to the $L^2$-norm on the unit disk, which is
also new. We also present numerical results consistent with the
theoretical analysis. As an application, we explain a connection
between the $\lambda$-RQM functions and the time-dependent solutions
of the imaginary-time wave equation in the elliptical coordinate
system.

We have relied heavily on \cite{McLachlan1951} as our primary
reference for Mathieu functions. While this text provides a vast
amount of information, culled from a series of articles in The London,
Edinburgh, and Dublin Philosophical Magazine and Journal of Science
from 1945 to 1951, it uses some confusing and inconsistent notation,
and the order of the topics and the mathematical rigor are rather
loose. Fortunately, several expositions have recently appeared,
notably \cite{ArfkenWeberHarris2013,BrCoZa2021} that clarify the
development of the fundamental properties of Mathieu
functions. However, as far as we know, the specific parameters in
$Z^{\pm}_{n,m}[\mu](\xi,\eta)$ for $\mu>0$ and the limit
$Z^{\pm}_{n,m}[0](x,y)$ are not easily found in the literature. We
hope that the results presented in Sections \ref{sec:preliminaries}
and \ref{Quaternionic_Mathieu_functions} regarding the completeness as
well as the limiting cases of Mathieu functions in terms of Bessel
functions will help clarify these issues for those interested in
special functions, even those who may not be concerned with the
quaternionic theory.


\section{Preliminaries\label{sec:preliminaries}} 

\subsection{Elliptical coordinates\label{subsec:ellipcoord}} 

Mathieu functions originated in the expression of operators involving
the two-dimensional Laplacian $\Delta=\partial^2/\partial x^2+\partial^2/\partial y^2$
in an equivalent formulation in \textit{elliptical coordinates}
\begin{equation} \label{eq:coordinates} x=\cosh \xi\cos \eta ,
\quad y=\sinh \xi\sin \eta ,
\end{equation}
where $\xi\ge0$ and $0 \le \eta \le 2\pi$, which parametrizes points
$(x,y)\in\R^2$. We will write
\begin{equation} \label{eq:Rectangle}
 R_{\xi_0} = [0,\xi_0) \times [0,2\pi]
\end{equation}
for the set of parameters corresponding to the open elliptical domain
\begin{equation}
  \Omega_{\xi_0} = \{(x, y) \in \R^2 \colon (\xi, \eta) \in R_{\xi_0} \},
\end{equation}
whose major and minor axes are of length
$2\cosh\xi_0$ and $2\sinh\xi_0$, respectively. We will freely identify
points $(x,y)\in\R^2$ with complex numbers $z\in\C$ and thus consider
$\Omega_{\xi_0} \subseteq \C$.

Note that points $(x,0)\in\R^2$ have a double representation
corresponding to $\xi=0$, namely $x=\cos\eta=\cos(2\pi-\eta)$, where
for small $\epsilon>0$, $(x,\epsilon)$ has coordinate $0\le\eta<\pi$
while a nearby $(x,-\epsilon)$ has coordinate $\pi\le\eta\le2\pi$.
Thus a differentiable real-valued function $F(x,y)$ defined in a
neighborhood of the segment $[-1,1]\times\{0\}$ corresponds via
\eqref{eq:coordinates} to a function $f(\xi,\eta)$ in a rectangle
$[0,\epsilon) \times [0,2\pi]$ satisfying the following properties:
 \begin{itemize}
\item[(a)] \textit{periodicity:} $f(\xi,0) = f(\xi,2\pi)$, 
\item[(b)] \textit{continuity of  displacement:}
  $f(0,\eta) = f(0,2\pi-\eta)$, and
\item[(c)] \textit{continuity of gradient:}
$\displaystyle \frac{\partial f}{\partial \xi}(0,\eta) =
-\frac{\partial f}{\partial \xi}(0,2\pi-\eta)$.
\end{itemize}

\begin{Definition} \label{defi:S}
We denote by $\S(\xi_0)$ the set of all $C^1$ functions in the
half-open rectangle $R_{\xi_0}$ satisfying the symmetries (a), (b), and
(c).  
 \end{Definition}

For functions whose first partial derivatives with respect to
$\xi$ and $\eta$ lie in  $\S(\xi_0)$, the Laplace operator
is expressed in elliptical coordinates by
\eqref{eq:coordinates} as
\begin{equation} \label{eq:defLmu}
\L = \frac{2 }{(\cosh 2\xi -
\cos 2\eta)}\big(
\frac{\partial^2}{\partial\xi^2}+\frac{\partial^2}{\partial\eta^2}
\big).
 \end{equation}
\subsection{Mathieu functions}

Mathieu functions were introduced in this Journal in \cite{Mathieu1868}. We summarize their basic properties following
the construction of McLachlan \cite{McLachlan1951}. For $q,a\in\R$, define $\phi^\pm(\eta)=\phi^\pm(\eta;\,q,a)$, $\psi^\pm(\xi)=\psi^\pm(\xi;\,q,a)$ as solutions of
the second-order ordinary differential equations
\begin{align}
  \phi''(\eta) + (a-2q\cos2\eta)\,\phi(\eta)=  0, \label{eq:mathieu1} \\[0.5ex]
  \psi''(\xi) - (a-2q\cosh2\xi)\,\psi(\xi)=  0, \label{eq:mathieu2}
\end{align}
where $\phi^+, \psi^+$ are even and $\phi^-, \psi^-$ are odd.
One combines the equation \eqref{eq:mathieu1} with either
\begin{align} \label{Boundary_condition1}
&  \phi^+(0) = \phi^+(2\pi),\quad (\phi^+{})^{\prime}(0) = (\phi^+{})^{\prime}(\pi) = 0,  
\end{align}
or
\begin{align} \label{Boundary_condition2}
& \phi^-(0) = \phi^-(2\pi),\quad \phi^-{}(0) = \phi^-{}(\pi) = 0,
\end{align}
to give a boundary-value problem in which the quantity $a$ is regarded
as an eigenvalue. These problems determine the
increasing sequences of real eigenvalues $a_n(q)$ or $b_n(q)$ (\textit{Mathieu characteristics}) indexed by $n=0,1,\dots$. The (angular) Mathieu functions are
then defined as
\begin{equation} \label{Mathieu_functions_trigonometric}
\ce_n(\eta,q) = \phi^+(\eta;\,q,a_n(q)), \quad \se_n(\eta,q) = \phi^-(\eta;\,q,b_n(q)),
\end{equation}
and the modified (or radial) Mathieu functions are
\begin{equation} \label{Mathieu_functions_hyperbolic}
\Ce_n(\xi,q) = \ce_n(i\xi,q), \quad \Se_n(\xi,q) = -i\se_n(i\xi,q).
\end{equation}
By the boundary conditions
\eqref{Boundary_condition1}-\eqref{Boundary_condition2}, $\ce_n$ is
even, $\se_n$ is odd, and both are $2\pi$-periodic for every fixed
$q$. We normalize these eigenfunctions so that the $L^2$-norms of
$\ce_n$ and $\se_n$ over the interval $[0,2\pi]$ are equal to
$\sqrt{\pi}$ except that $\ce_0 \equiv 1$ (cf.\ the discussion in
\cite{McLachlan1951} of why it is preferable not to choose certain
Fourier coefficients to be equal to $1$ as is done in
\cite{WhittakerWatson1927}). Note that $\Ce_n(\xi,q)$ and
$\Se_n(\xi,q)$ are even and odd real-valued solutions of
\eqref{eq:mathieu2}, respectively. Further references on Mathieu
functions, including numerical aspects, are
\cite{Alhargan1996,Alhargan2001,Gutierrez-VegaIturbe-CastilloChavez-Cerda2000,HabashyKongChew1986,LinderFreese1994,MeixnerSchaefke1954,PhelpsHunter1965,Sarchinger1894,ToyaymaShogen1994}.

To unify the notation, we will write
\begin{align}  \label{eq:psiphi}
\psi^+_n[q](\xi)&=\Ce_n(\xi,q),\ \
  \psi^-_n[q](\xi)= \Se_n(\xi,q),
 \nonumber\\[0.25ex]
  \phi^+_n[q](\eta)&= \ce_n(\eta,q),\ \ \,
  \phi^-_n[q](\eta)=\se_n(\eta,q).
\end{align}

\begin{Definition}\label{def:zeta}  For $n \geq 0$, the even and odd
  \textit{two-dimensional Mathieu functions} corresponding to
  $q>0$ are the products
  $\zeta^\pm_n[q]$ given by
\begin{align*}
 \zeta^+_n[q](\xi,\eta) &= \psi^+_n[q](\xi) \,\phi^+_n[q](\eta),   \\[0.25ex]
  \zeta^-_n[q](\xi,\eta)&= \psi^-_n[q](\xi)\,\phi^-_n[q](\eta).  
\end{align*}
\end{Definition}

We never use $\psi^-_0$, $\phi^-_0$, $\zeta^-_0$ since they are
identically zero. It can be observed that $\psi^+_n[q] \, \phi^+_n[q]$
and $\psi^-_n[q]\,\phi^-_n[q]$ satisfy the conditions (a), (b), and
(c) of $\S_0(\xi_0)$ and are in fact the only possible products of
Mathieu and modified Mathieu functions satisfying these properties
simultaneously.

 These classical functions have been considered for example in
 \cite{ArfkenWeberHarris2013,McLachlan1945,McLachlan1947,McLachlan1951,Sato2006,Sato2010}. This
 is because the Mathieu equations arise from the Helmholtz equation
\begin{align}   \label{eq:helmholtz}
    \Delta F + \lambda^2 F = 0
\end{align}
by the method of separation of variables:

\begin{Proposition} \label{prop:zetahelmholtz}
  The two-dimensional Mathieu functions (with $q>0$) are solutions of the
  Helmholtz equation in elliptical coordinates, that is,
\begin{equation} \label{eq:mathieuelliptic}
\left(\L +  4q \right)\zeta^{\pm}_n[q] =0,
\end{equation}
where $\L$ is defined by \eqref{eq:defLmu}.
\end{Proposition}

We have
\begin{Proposition} \label{prop:orthogonalmathieu} Let $q > 0$.  (i)
  The collection $\{\phi^\pm_n[q]\}_{n=0}^\infty$ is a complete subset
  of $C^0[0,2\pi]$ in the sense of uniform convergence and is a
  complete orthogonal system in $L^2([0,2\pi])$.  \\(ii) For every
  $\xi_0>0$, the collection $\{\psi^\pm_n[q]\}_{n=0}^\infty$ is a
  complete subset of $C^0[0,\xi_0]$ in the sense of uniform
  convergence and is a complete (nonorthogonal) system in
  $L^2([0,\xi_0])$. 
\end{Proposition}

\begin{proof}
  Statement (i) is proved in \cite[p.\ 104]{Higgins1977}.  Statement (ii), which
  is often used in the literature without formal justification, can be
  deduced, for example, from equations (5), (6) on page 207 of
  \cite{McLachlan1951}, which allow one to express
  $\cos(2\sqrt{q}\cosh\xi)$ and $\sin(2\sqrt{q}\cosh\xi)$ in terms of
  $\{\psi^\pm_{n}[q]\}$. One applies elementary observations such as
  the fact that uniform convergence is not affected by
  reparametrization and that convergence in $L^\infty$ implies
  convergence in $L^2$. 
\end{proof}

We will often use the weighted scalar product
\begin{equation} \label{eq:weighted}
  \langle f,g\rangle_{\xi_0} = \int_0^{2\pi}\!\!\!\int_0^{\xi_0} f(\xi,\eta)\,g(\xi,\eta)\, \frac{1}{2} (\cosh2\xi-\cos2\eta) \, d\xi d\eta.
\end{equation}  It is readily verified
that when changing from elliptical to Cartesian coordinates, we write
$ F(x,y)=f(\xi,\eta)$, $G(x,y) = g(\xi,\eta)$, and so \eqref{eq:weighted} is equal to the unweighted inner product of
$F, G$ in $L^2(\Omega_{\xi_0})$ with respect to Lebesgue measure.
 
\subsection{Quaternions and $\lambda$-metamonogenic functions}

 We are interested in a theory of functions from the plane domain
 $\Omega_{\xi_0}$ to $\R^3$. For this purpose, we embed $\R^3$ into the
space $\H$ of (real) quaternions by considering quaternions with
vanishing $\textbf{k}$-coefficient. These are also known as
\textit{reduced quaternions}. A quaternion
\cite{K2003,MoraisGeorgievSproessig2014} is notated as
$a =a_0+\ei a_1+\ej a_2+\ek a_3$, where $a_0 = \Sc a$ is the scalar part of $a$. Here $a_m=[a]_m\in\R$ and
$\ei,\ej,\ek$ are the quaternionic imaginary units satisfying
$\ei^2 = \ej^2 = \ek^2 = \ei \ej \ek = -1$. The usual
component-wise defined addition is implied when $\H$ is identified with
$\R^4$, which also induces the absolute value on $\H$.

Let $\Omega$ be a domain in $\mathbb{R}^2$ (open and connected). Let
$L^2(\Omega,\R^3)$ denote the $\R$-linear space of all $\R^3$-valued
functions $F\colon\Omega\to\R^3$, that is
\begin{eqnarray*}
F(x,y) = [F]_0 (x,y) + \ei [F]_1(x,y) +
    \ej [F]_2(x,y), 
\end{eqnarray*}
such that the components $[F]_m$ ($m=0,1,2$) are in the usual
$L^2(\Omega,\R)$.  We will use the $\R$-valued inner product
\begin{eqnarray} \label{eq:InnerProduct}
 \langle F, G \rangle_{\Omega} =
  \Sc\int \!\!\!\int_{\Omega} \;  \overline{F(x,y)} \, G(x,y) \,dx\,dy
\end{eqnarray}
for $F, G \in L^2(\Omega,\R^3)$. It is a real Hilbert space with the
associated norm $\|F\|_{\Omega} = \langle F, F\rangle^{1/2}$, and it coincides
with the usual $L^2$-norm for $F$, viewed as a vector-valued function
in $\Omega$ \cite{GS1989,GS1997}. Similarly, one has the weighted
inner product $\langle f, g \rangle_{\xi_0}$ analogous to
\eqref{eq:weighted}, \eqref{eq:InnerProduct} for $\R^3$-valued
functions defined in the coordinate rectangle $R_{\xi_0}$.

We consider the operator
\begin{equation} \label{eq:defD}
D = \ei\frac{\partial}{\partial x} + \ej\frac{\partial}{\partial y}.
\end{equation}

\begin{Definition} \label{Main_definitions} Given
  $\lambda \in \R \setminus \{0\}$, a function of two variables is
  said to be \textit{$\lambda$-metaharmonic} if it is in the kernel of
  the Helmholtz operator $\Delta+\lambda^2$ and
  \textit{$\lambda$-metamonogenic} 
  if it is in the kernel of the Moisil-Teodorescu operator $D+\lambda$. For domains
  $\Omega\subseteq\R^2$, we write
  \begin{equation*}
\M(\Omega;\lambda) = \Ker (D + \lambda) \subseteq
  C^1(\Omega,\R^3)
  \end{equation*}
and
$\M_2(\Omega;\lambda) = \M(\Omega;\lambda) \cap
L^2(\Omega,\R^3)$. Further, for $\xi_0>0$, we write
\begin{equation*}
\M({\xi_0};\lambda) = \Ker (\D + \lambda) \subseteq \S(\xi_0),
\end{equation*}
where
$\S(\xi_0)$ was given in Definition \ref{defi:S}, and
$\M_2(\xi_0;\lambda) = \M({\xi_0};\lambda) \cap L^2(R_
{\xi_0},\R^3)$ with the weighted product \eqref{eq:weighted}.
\end{Definition}

We will see in Proposition \ref{prop:convergence} that
$\M_2(\xi_0;\lambda)$ is a closed subset of $L^2(R_{\xi_0},\R^3)$ and
therefore is a Hilbert space.
 
While $D+\lambda$ does not generally commute with
quaternionic functions, since we are considering $\R^3$-valued
functions in this paper, the definition of $\lambda$-metamonogenic function does not depend on
whether one applies $D+\lambda$ from the left or the right.

\begin{Proposition}  \label{prop:factorhelmholtz}
$(D+\lambda)(D-\lambda) = -(\Delta+\lambda^2)$ for $\lambda \in \R \setminus \{0\}$.
\end{Proposition}

The factorization of the Laplacian via $D$ was observed by W.\
Hamilton \cite[Section CVII]{WRH1853}\footnote{We are grateful
to V.\ V.\ Kravchenko for pointing out this fact to us.}. The factorization for the
Helmholtz operator appears in modern literature in
\cite{Gurl1986,GuerlebeckHabethaSproessig2008,GuerlebeckHabethaSproessig2016}
for functions defined in domains in $3$-space. It was subsequently
explored in \cite{KrSh,K2003} and in elliptical, spherical, and ellipsoidal
(prolate and oblate spheroidal) coordinate systems in
\cite{LuPeRoSha2013,LunaMoraisRosaShapiro2016,MoraisRosaKou2015,MoraisRosa2015,MoraisPerezLe2016},
where a version of $D+\lambda$ appears.

\begin{Proposition} \label{prop:convergence} Let
  $\Omega\subseteq\R^2$. (i) Let $F_n\in \M({\Omega;\lambda})$ and
  suppose that $F_n\to F$ uniformly on compact subsets of
  $\Omega$. Then $F\in \M({\Omega;\lambda})$.  (ii) Let
  $F_n,F\in \M_2({\Omega;\lambda})$ and suppose that
  $\|F_n\to F\|_2\to0$. Then $F_n\to F$ uniformly on compact subsets
  of $\Omega$.
\end{Proposition}

\begin{proof}
  (i) The function $\Theta_\lambda(z)=(1/2)Y_0(\lambda|z|)$ is a fundamental
  solution of the Helmholtz equation \eqref{eq:helmholtz}, where $Y_0$
  is the Bessel function of the second kind of order zero \cite[p.\
  154]{Vladimirov1971}. (We have combined the complex-valued
  fundamental solutions generally treated in the literature into a
  real-valued solution since we are not interested in the behavior at
  infinity.) Define
  \begin{equation*}
  K_\lambda(z) = -(D-\lambda)\Theta_\lambda(z),
  \end{equation*}
  which is smooth on
  $\R^2\setminus\{0\}$ and satisfies $(D+\lambda) K_\lambda(z)=0$ for
  $z\not=0$.  From $(D+\lambda)F_n=0$, standard techniques
  \cite{K2003} produce the Cauchy integral formula for
  $\M({\Omega;\lambda})$,
  \begin{equation*}
  F_n(z_0) = \int_C K_\lambda(z-z_0) \vec n(z) F_n(z)\,ds
  \end{equation*}
  for every simple closed curve $C$ encircling $z_0$ in $\Omega$. Here
  $ds$ represents arc length and $\vec n(z) = \ei n_1(z) + \ej n_2(z)$,
  where $n_1(z)+in_2(z)$ is the outward unit normal vector to $\C$ at
  $z$. (The analogous integral kernel in elliptical coordinates is
  worked out in \cite{LuPeRoSha2013}.) By uniform convergence on $C$,
  we have a Cauchy integral representation for $F(z)$ as well, from
  which it follows that $(D+\lambda)F=0$.  \\(ii) We have explicitly
  \begin{equation*}
  2K_\lambda(z) = \lambda Y_0(\lambda|z|)+
    \ei \frac{\lambda\re (z)}{|z|} \, Y_1(\lambda|z|) +
    \ej \frac{\lambda\im (z)}{|z|} \, Y_1(\lambda|z|),
  \end{equation*}
where $Y_1=-Y_0^{\prime}$ is the Bessel function of the second kind of order one. It is well known that
 \begin{equation*}
 \lim_{r\downarrow0}\frac{Y_0(r)}{\log r} =
    -\lim_{r\downarrow0}rY_1(r) = \frac{2}{\pi} .
 \end{equation*}
 By the Cauchy formula, for any $G\in\M(\Omega;\lambda)$ we have
 \begin{equation*}
 G(z_0) = \int_0^{2\pi} K_\lambda(re^{i\theta}) \, \vec n(re^{i\theta})
   \, G(re^{i\theta})\,r d\theta 
 \end{equation*}
 as long as $r$ does not exceed the distance $r_0$ from $z_0$ to the
 boundary of $\Omega$.  Multiply by $r$ and integrate with respect
 to $r$ to obtain
 \begin{equation*}
 \frac{r_0^2}{2} \le  c_{r_0} \int_0^{2\pi} \hspace{-0.075cm} \int_0^{r_0} |G(z)|\,dx\,dy,
 \end{equation*}
 where the bound $|r K_\lambda(r e^{i\theta})|\le c_{r_0}$ for
 $0\le r\le r_0$ is guaranteed by the explicit representation of
 $K_\lambda$. From the Cauchy-Bunyakovsky-Schwarz inequality, we obtain
 an estimate
 \begin{equation*}
 |G(z)| \le c_{r_0}' \|G\|_2,
 \end{equation*}
 valid for $|z-z_0|\le r_0$. Apply this to $G=F-F_n$ to deduce that
 $F_n$ converges locally uniformly to $F$, hence uniformly on every
 compact subset of $\Omega$.  
\end{proof}

It is clear that the change of coordinates \eqref{eq:coordinates}
defines one-to-one correspondences
$\M(\Omega_{\xi_0};\lambda)\leftrightarrow \M(\xi_0;\lambda)$ and
$\M_2(\Omega_{\xi_0};\lambda)\leftrightarrow \M_2(\xi_0;\lambda)$,
the latter being a Hilbert space isometry, and that Proposition
\ref{prop:convergence} has an immediate analogue for functions in
$\S(\xi_0)$. Note that compact subsets of $\Omega_{\xi_0}$ correspond
to compact subsets of $R_{\xi_0}$ despite the ambiguity of
the correspondence along $[-1,1]$ noted in Section \ref{subsec:ellipcoord}.

We conclude this section with the following technical lemma, whose proof is immediate from Definition \ref{Main_definitions}.
\begin{Lemma} \label{lemm:zeroscalar}
Let $F \in \M({\Omega;\lambda})$. If $F_0$ is constant, then $F = 0$ identically.
\end{Lemma}


\section{Quaternionic Mathieu functions} \label{Quaternionic_Mathieu_functions}

In this section, we introduce the reduced quaternionic analogue of the
Mathieu functions.  Recalling \eqref{eq:defLmu}, we have that the
Helmholtz equation is represented in elliptical coordinates as
\begin{align} \label{eq:helmholtzmu}
  (\L   + \lambda^2)h(\xi,\eta)= 0    
\end{align}
for $h \in C^2(R_{\xi_0},\R)$. We look at the similar
expression for $D+\lambda$ and give a complete system for its
kernel.

\subsection{$\lambda$-metamonogenic functions in elliptical coordinates}

An elementary calculation gives the following.
\begin{Proposition}[\cite{LuPeRoSha2013}] \label{prop:Dell} The
  operator $D$ in elliptical coordinates in $\Omega_{\xi_0}$ is given by
\begin{align*} 
\D  &=   \frac{2}{ (\cosh 2\xi - \cos 2\eta)}
 \bigg(  (\ei \sinh \xi \cos\eta + \ej \cosh \xi \sin \eta) 
 \frac{\partial}{\partial \xi}  \\
& \hspace*{21ex} + \;  (\ej \sinh \xi \cos \eta - \ei \cosh \xi \sin\eta)
 \frac{\partial}{\partial \eta}\bigg).   
\end{align*}
More precisely, if $f(\xi,\eta)$ is the expression of $F(x,y)$ in
$\Omega_{\xi_0}$ for a fixed $C^1$ function $F$, then
$\D f(\xi,\eta) = DF(x,y)$. 
\end{Proposition}
 
The following statement follows directly from the definition of $\D$
and Proposition \ref{prop:factorhelmholtz}.

\begin{Proposition} \label{prop:Factorization_Laplace_Operator} The
  operator $\D$ factors the Helmholtz operator in
  elliptical coordinates; that is,
\begin{equation} \label{eq:Factorization} 
  -(\D+\lambda)(\D-\lambda) \, = \,
  -(\D-\lambda)(\D+\lambda)   =   \L + \lambda^2,
\end{equation}
where $\L$ is defined by \eqref{eq:defLmu}.
\end{Proposition}

The factorization \eqref{eq:Factorization} suggests that reduced
quaternionic $\lambda$-meta\-mo\-no\-genic functions should play a
role for the operator $\L$, similar to the usual metaharmonic
functions in two variables for the corresponding Helmholtz operator
\cite{Vekua1968}. However, there is a significant difference, inasmuch
as the first-order operator $\D$ has nonconstant coefficients, and the
theories of such operators are much more sophisticated (see, e.g.,
\cite{KrSh}).


\subsection{$\lambda$-RQM functions\label{subsec:RQMfunctions}}

In this part, we define the $\lambda$-reduced quaternionic Mathieu functions and
express their components in terms of the solutions of the angular and
radial Mathieu equations. Such $\R^3$-valued Mathieu functions were
considered in \cite{LuPeRoSha2013}, and properties of the Cauchy
integral and related transformations were explored. Here we are
interested in constructing explicit formulas.
 
\begin{Definition} \label{Definition_RQM functions} Let $\lambda\in\R\setminus\{0\}$. For $n \geq 0$, the corresponding  
  \textit{$\lambda$-reduced quaternionic Mathieu functions} ($\lambda$-RQM functions) are the  $\R^3$-valued functions
  $\Mat^+_n[\lambda]$, $\Mat^-_n[\lambda]$ defined in the
  parametric space $[0,\infty) \times [0,2\pi]$ by
\begin{equation*}
  \Mat^\pm_n[\lambda] = -\frac{1}{\lambda} \,
  (\D -\lambda)\zeta^\pm_n[\lambda^2/4],
\end{equation*}
where $\zeta^\pm_n[q]$ are given in Definition \ref{def:zeta}.
\end{Definition}

Note that both the source function and the operator are adapted to the elliptical domain of interest in this definition.

The quaternionic components of the $\lambda$-RQM functions can be read off
immediately from the formulas of Definition \ref{def:zeta} and
Proposition \ref{prop:Dell}:

\begin{Proposition} \label{prop:mathieucomponents} Let $\lambda\in\R\setminus\{0\}$.
  The $\lambda$-RQM functions can be expressed as follows, where
  we write $q=\lambda^2/4$:
\begin{align*}
  & \Mat^\pm_n[\lambda](\xi,\eta) =
     \psi^\pm_n[q](\xi)\phi^\pm_n[q](\eta)  -
 \frac{2}{\lambda (\cosh 2\xi - \cos 2\eta)} \times \\
& \ \ \Bigl( \ei \big( \sinh \xi \cos \eta \, \psi^\pm_n[q]'(\xi) \phi^\pm_n[q](\eta) - \cosh \xi \sin \eta \, \psi^\pm_n[q](\xi) \phi^\pm_n[q]'(\eta) \big) \\
& \ + \ej\big( \cosh \xi \sin \eta \, \psi^\pm_n[q]'(\xi) \phi^\pm_n[q](\eta) + \sinh \xi \cos \eta \, \psi^\pm_n[q](\xi) \phi^\pm_n[q]'(\eta) \big) \Bigl).
\end{align*}
In particular, the scalar parts are simply $\Sc \Mat^\pm_n[\lambda]  =  \zeta^{\pm}_n[\lambda^2/4]$.
Further, $\Mat^\pm_n[\lambda] \in \M_2(\xi_0;q)$ for every $\xi_0>0$.
\end{Proposition}

The final statement, that the $\Mat^\pm_n[\lambda]$ are
$\lambda$-metamonogenic, is seen from the factorization of
Proposition \ref{prop:factorhelmholtz} and the fact that the
$\zeta^\pm_n[q]$ are always $2\sqrt{q}$-metaharmonic as well as
satisfying the symmetry properties of $\S(\xi_0)$. Note that 
$\Mat^\pm_n[-\lambda]=\overline{\Mat^\pm_n[\lambda]}$. The
idea of completing a scalar solution of the Helmholtz equation
to a $\lambda$-metamonogenic function by applying
$(-1/\lambda)(D-\lambda)$ may be found in \cite{DelgadoKravchenko2019}.

\begin{figure}[t]
  \pic{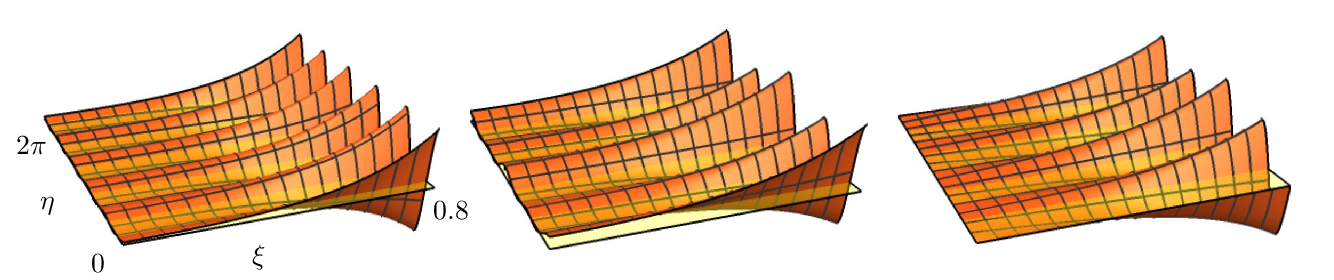}{1,-5.7}{3}{scale=.28}\\
  \pic{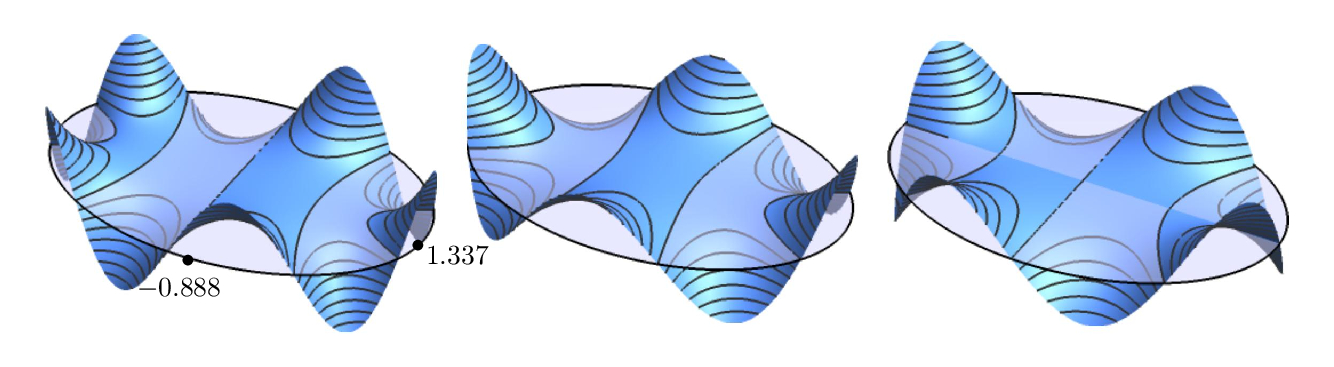}{1,-5.7}{3}{scale=.28}
  \caption{Graphs of the components of the RQM function
    $\Mat^+_{5}[1.5]$ over the coordinate rectangle $R_{0.8}$ (above)
    and the corresponding 1.5-metamonogenic function in the ellipse
    $\Omega_{0.8}$ (below). \label{fig:rqm}}
\end{figure}


\subsection{Completeness}

We now prove a reduced-quaternionic counterpart of Proposition \ref{prop:orthogonalmathieu}.
\begin{Theorem} \label{Theorem_completeness}
For fixed $\lambda\in\R\setminus\{0\}$, the closure of the collection of $\lambda$-RQM functions $\{\Mat^\pm_n[\lambda]\colon n\ge0\}$, is a complete subset of
  $\M({\xi_0};\lambda)$ in the sense of uniform convergence on
  compact subsets. Further, it is complete in the Hilbert space
  $\M_2(\xi_0;\lambda)$.
\end{Theorem}

\begin{proof}
  Let $f\in \M(\xi_0;\lambda)$ be arbitrary and write
  $f=f_0+\ei f_1+\ej f_2$. Then by Proposition
  \ref{prop:orthogonalmathieu}, we may approximate the function
  \begin{equation*}
 f_0 = \sum_{n=0}^\infty a^\pm_n \zeta^\pm_n[q]
 \end{equation*}
 as a real linear combination of the scalar parts of the
 $\Mat^\pm_n[\lambda]$, where $q=\lambda^2/4$, the sum converging uniformly
 on $[0,\xi_0-\epsilon)\times[0,2\pi]$ for every $\epsilon>0$. Define
 \begin{equation} \label{Function_construction}
 g = -\frac{1}{\lambda}(\D-\lambda)f_0,
 \end{equation}
 so $g=g_0+\ei g_1+\ej g_2 \in\M(R_{\xi_0};\lambda)$ and $g_0=f_0$. By Lemma
 \ref{lemm:zeroscalar}, $g=f$. By Proposition \ref{prop:convergence}(i),
 \begin{equation*}
 g_1 = -\frac{1}{\lambda}\sum_{n=0}^\infty a_n^\pm
   [(D-\lambda)\zeta^\pm_n[q]]\,]_1 = \sum_{n=0}^\infty a_n^\pm [\Mat^\pm_n[\lambda]]_1
  \end{equation*}
   converging uniformly on compact subsets. A similar formula holds for $g_2$. Therefore
 \begin{equation*}
 g = \sum_{n=0}^\infty a_n \Mat^\pm_n[\lambda].
 \end{equation*}
 Thus $f$ is approximated uniformly on compact subsets.

 The completeness in $L^2$ is proved similarly, using  
Proposition \ref{prop:convergence}(ii). 
\end{proof}

\subsection{Limiting behavior of RQM-functions viewed in a disk} \label{subsec:limiting_RQMfunctions}

We consider now the approximation of a disk by ellipses, and the corresponding behavior of the RQM-functions.
Let $0<\mu<1$, and define $\xi_\mu$ by
\begin{align} \label{eq:defximu}
   \mu\cosh\xi_\mu=1.
\end{align}
Since $\xi_\mu\to\infty$ as $\mu\to0$, the corresponding elliptical
domain becomes arbitrarily large, so we rescale it by a constant
factor and consider
\begin{equation*}
\mu \Omega_{\xi_\mu} = \{ \mu z \colon z\in  \Omega_{\xi_\mu} \}.
\end{equation*}
Then it is readily seen that
$\mu_1\Omega_{\xi_{\mu_1}}\subseteq \mu_2\Omega_{\xi_{\mu_2} }$ when
$\mu_1\le\mu_2$, and that
\begin{equation*}
\bigcup_{0<\mu<1}  \mu \Omega_{\xi_\mu} = B_1(0),
\end{equation*}
where $B_1(0)=\{z\colon |z|<1\}$ is the unit disk.

For fixed $w=u+iv\in B_1(0)$, let $z=w/\mu$, which is in
$\Omega_\mu$ for sufficiently small $\mu$. One may see from
the relation $u^2+v^2+\mu^2=\mu^2(\cosh^2\xi+\cos^2\eta)$ that the
limiting values of $\mu\cosh\xi$ and $\mu\sinh\xi$ as $\mu\to0$ is $|w|$, and thus  
\eqref{eq:coordinates}  indeed tends to
$u=|w|\cos (\arg w)$, $v=|w|\sin(\arg w)$.

Before proceeding with the asymptotic analysis of the
$ \Mat^\pm_n[q]$, we will need some asymptotic relations of their
component functions.  In \cite[pp.\ 368--369]{McLachlan1951}, the
following expressions are used,
\begin{align*}
  p_{2l}'(q) &= \frac{ (-1)^l\ce_{2l}(0,q)\ce_{2l}(\pi/2,q)}
               {A_0^{(2l)}(q)},  \nonumber \\
  p_{2l+1}'(q) &= \frac{ (-1)^{l+1}\ce_{2l+1}(0,q)\ce_{2l+1}'(\pi/2,q)}
               {\sqrt{q}A_1^{(2l+1)}(q)},  \nonumber\\
  s_{2l+1}'(q) &= \frac{ (-1)^{l}\se_{2l+1}'(0,q)\se_{2l+1}(\pi/2,q)}
               {\sqrt{q}B_1^{(2l+1)}(q)}, \nonumber \\
  s_{2l+2}'(q) &= \frac{ (-1)^l\se_{2l+2}'(0,q)\se_{2l+2}'(\pi/2,q)}
               {qB_0^{(2l+2)}(q)}, 
\end{align*}
for $\l\ge0$. (Here the primes do not refer to derivatives.) We do not need the definitions of the $A_m^{(n)}(q)$,
$B_m^{(n)}(q)$ as Fourier coefficients of $\ce_n$, $\se_n$, but only the
asymptotic relations \cite[p.\ 185]{McLachlan1951}:
\begin{align*}
  A_0^{(2l)}(q) &\sim \frac{q^l}{2^{2l-1}(2l)!}, \quad
    A_1^{(2l+1)}(q) \sim \frac{q^l}{2^{2l}(2l)!}, \\
  B_1^{(2l+1)}(q) &\sim \frac{q^l}{2^{2l}(2l)!}, \quad
  B_2^{(2l+2)}(q) \sim \frac{(l+1)q^l}{2^{2l}(2l+1)!},
\end{align*}
as $q\to0$, for all $l\ge0$, except (not noted in
\cite{McLachlan1951}) that $A_0^0\sim1$. Let $\beta_0=1/\sqrt{2}$ and
for $n\ge1$, let
\begin{align}
  \beta_n = 2^{n-1}n!.
\end{align}

\begin{Lemma}\label{lemm:asympgps}
  $q^{n/2}p_n'(q)\to \beta_n$ and $q^{n/2}s_n'(q)\to \beta_n$ as
  $q\to0$.  
\end{Lemma}
\begin{proof}
  It is shown in \cite[p.\ 367]{McLachlan1951} that for fixed
  $w\in B_1(0)$, with the elliptical coordinates
  \eqref{eq:coordinates} depending on $\mu$ through $(x,y)$,
\begin{align*}
  \ce_n(\eta,q) &\to \cos(n\arg w),\quad \ce_n'(\eta,q) \to -n\sin(n\arg w),\\
  \se_n(\eta,q) &\to \sin(n\arg w),\quad \se_n'(\eta,q)  \to n\cos(n\arg w)
\end{align*}
as $q \to 0$. (The derivatives are with respect to the first variable.) Thus we have for $l>0$
\begin{equation*}
  p_{2l}'(q) \sim \frac{(-1)^l \cos(2l\,0) \cos(2l\,\pi/2)}{q^l/(2^{l-1}(2l)!)} = \frac{2^{2l-1}(2l)!}{q^l} = \beta_{2l}\,q^{-l}
\end{equation*}
as required.

The remaining cases are similar.
\end{proof}

As a final notational convenience, we will write
\begin{equation*} 
   \Phi^+_n(w) = \cos(n\arg w), \quad \Phi^-_n(w) = \sin(n\arg w).
\end{equation*}

\begin{Proposition} \label{prop:psiphilimits} Let $w\in B_1(0)$ be
  fixed, and for $0<\mu<1$, let $z=z(\mu)=w/\mu$ correspond to
  coordinates $(\xi,\eta)=(\xi(w,\mu),\eta(w,\mu))$ under
  \eqref{eq:coordinates}. Then $w\in\mu\Omega_{\xi_\mu}$ for
  sufficiently small $\mu$ (so that $\xi, \eta$ are defined). Suppose
  that $q=q(\mu)>0$ and $q(\mu)\to0$ as $\mu\to0$ in such a way that
  the finite positive limit
  $\alpha = \lim_{\mu\to0} 2\sqrt{q(\mu)}/\mu$ exists. Then
  \begin{align*}
    \phi^\pm_n[q](\eta(w,\mu)) &\to \Phi^\pm_n(w), \\
    \phi^\pm_n[q]'(\eta(w,\mu)) &\to \mp n\Phi^\mp_n(w) ,\\
   q^{n/2}\psi^\pm_n[q](\xi(w,\mu)) &\to
   \beta_n J_n(\alpha|w|), \\
   q^{n/2}\psi^\pm_n[q]'(\xi(w,\mu)) &\to
    \alpha \beta_n |w| J_n'(\alpha|w|)
 \end{align*}
 as $\mu\to0$, where $J_n$ is the Bessel function of the first kind of order $n$.
\end{Proposition}

\begin{proof}
  The limits for $\phi^\pm_n$ were already remarked in the proof of
  Lemma \ref{lemm:asympgps}.  It was shown in \cite[p.\
  368]{McLachlan1951} that under our assumptions,
  \begin{align*}
   q^{n/2}\Ce_n(\xi,q) &\to \beta_n J_n(\alpha|w|), \\   
   q^{n/2}\Ce_n'(\xi,q) &\to \beta_n\alpha|w| J_n'(\alpha|w|), \\    q^{n/2}\Se_n(\xi,q) &\to \beta_n J_n(\alpha|w|), \\   
   q^{n/2}\Se_n'(\xi,q) &\to \alpha \beta_n |w| J_n'(\alpha|w|). 
  \end{align*}
 Thus we have the asserted asymptotic values of
  $\phi^\pm_n$, $\psi^\pm_n$, and their derivatives.
\end{proof}

We remark that some authors \cite{McLachlan1945,Sato2006,Sato2010}
have written the limit in a loose way, such as in
``$\Ce_n(\xi,q)\to p_n'J_n(\alpha|w|)$,'' leading to some loss of
clarity in expressing the asymptotic behavior.
 
We return to our study of RQM-functions. As $\mu \to 0$, under
controlled convergence of the eigenvalue $\lambda$, we may
obtain a limit of RQMs.

\begin{Theorem} \label{theorem:mathieulimits} Let $q=q(\mu)$ be as in
  Proposition \ref{prop:psiphilimits}. Let
  \[F^\pm_n[\alpha;\mu](x,y) = q(\mu)^{n/2}
    \Mat^\pm_n[2\sqrt{q(\mu)}](\xi,\eta)\] under the change of
  variables \eqref{eq:coordinates}. When $\mu\to 0$, the
  $2\sqrt{q(\mu)}$-metamonogenic functions
  $F^\pm_n[\alpha;\mu](w/\mu)$ tend pointwise in $B_1(0)$ to
  the $\alpha$-metamonogenic limit $F^\pm_n[\alpha]$ defined by
\begin{align*}
 F^\pm_n[\alpha](w) &= \beta_nJ_n(\alpha|w|)\Phi^\pm_n(w)\\
&\ \ - \frac{\beta_n}{|w|} \, (\ei \re w + \ej \im w) J_n'(\alpha|w|)\Phi^\pm_n(w) \\
 &\ \ \mp \frac{n\beta_n}{\alpha |w|^2} \, (\ei \im w - \ej \re w) J_n(\alpha|w|)\Phi^\mp_n(w). 
\end{align*}
\end{Theorem}

Of course, if one replaces $\Mat^\pm_n[2\sqrt{q(\mu)}]$
with $\Mat^\pm_n[-2\sqrt{q(\mu)}]$, then $F^\pm_n[\alpha;\mu]$ (resp. $F^\pm_n[\alpha]$) will be replaced with $\overline{F^\pm_n[\alpha;\mu]}$ (resp. $\overline{F^\pm_n[\alpha]}$).

\begin{proof}  
By applying Proposition \ref{prop:psiphilimits} to Proposition \ref{prop:mathieucomponents}, it is immediate that the scalar parts of $F^\pm_n[\alpha](w)$ are as stated. For the vector parts, we recall that
$\mu\cosh\xi \to |w|$, $\mu\sinh\xi \to |w|$, $\mu^2\cosh2\xi\to2|w|^2$ as $\mu\to0$.  Thus for the common factor of the vector part,
\begin{equation*}
-\frac{1}{\mu \sqrt{q}(\cosh2\xi-\cos2\eta)}
  \to -\frac{1}{\alpha|w|^2}
\end{equation*}
while
\begin{equation*}
\ei \, q^{n/2} \big(\mu \sinh \xi \cos \eta \, \psi^{\pm}_n{\,}'(\xi) \phi^\pm_n(\eta) - \mu \cosh \xi \sin \eta \, \psi^\pm_n(\xi) \phi^\pm_n{\,}'(\eta) \big) \to
\end{equation*}
\begin{equation*}
\ei \big((\re w) \beta_n\alpha|w|J_n'(\alpha|w|)\Phi^\pm_m(w)
  - (\im w)\beta_nJ_n(\alpha|w|)(\mp n\Phi^\mp_m(w) \big),
  \end{equation*}
as well as
\begin{equation*}
\ej \, q^{n/2} \big(\mu \cosh \xi \sin \eta \, \psi^\pm_n{\,}'(\xi) \phi^\pm_n(\eta) + \mu \sinh \xi \cos \eta \, \psi^\pm_n(\xi) \phi^\pm_n{\,}'(\eta) \big)\to 
\end{equation*}
\begin{equation*}
\ej \big((\im w) \beta_n\alpha|z|J_n'(\alpha|w|)\Phi^\pm_m(w)
 + (\re w)\beta_nJ_n(\alpha|w|)(\mp n\Phi^\mp_m(w) \big).
\end{equation*}
After regrouping the terms, we have the limiting result.

To show that $F^\pm_n[\alpha]$ are $\alpha$-metamonogenic one can use a limiting argument or a direct computation as follows.  
\begin{align*}
\Sc \left((D + \alpha)F^\pm_n[\alpha]\right) =& \, \frac{\beta_n}{\alpha |w|^2} \, \Phi^\pm_n(w) \Bigl( (\alpha |w|)^2 J_n''(\alpha|w|) + \alpha |w| J_n'(\alpha|w|) \\
&+ ((\alpha |w|)^2 - n^2) J_n(\alpha|w|) \Bigr)\\
=& \, 0,
\end{align*}
which follows from Bessel's differential equation. One may verify similarly that $\VEC \left((D + \alpha)F^\pm_n[\alpha] \right) = 0$.
\end{proof}

An illustration of the approximation of $F^\pm_n[\alpha]$ by
$q(\mu)^{n/2} \Mat^\pm_n[2\sqrt{q(\mu)}]$ is given in Figure
\ref{fig:limitM}.

\begin{figure}[ht!]
  \pic{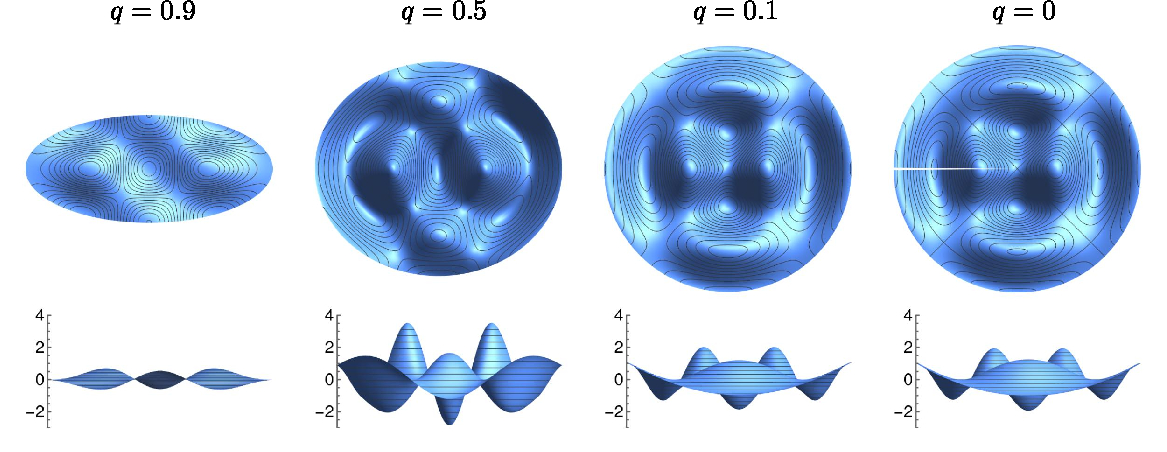}{0,-7.4}{4.3}{scale=.34}
  \caption{Illustration of the convergence of the scalar part
    (classical Mathieu functions) given by Proposition
    \ref{prop:psiphilimits}. \label{fig:limitM}}
\end{figure}

\section{Zero-boundary RQM-functions} \label{Zero-boundary_RQM-functions}

\subsection{Distinguished potentials $q^{\pm}_{n,m}(\xi_\mu)$}

In some applications, one is interested in planar Mathieu functions
which vanish on the boundary of the ellipse. This occurs in studies
of vibrating membranes \cite[p.\ 296]{McLachlan1951}, \cite[Ch.\
32]{ArfkenWeberHarris2013}, or of heat conduction in an elliptical
cylinder \cite{Sato2010}. The theory is somewhat different from what
we have seen so far because a complete system must be constructed with
functions corresponding to distinct eigenvalues $\lambda$.
 
It is shown in \cite{McLachlan1951} that for $\xi=\xi_0$ constant,
$\Ce_n(\xi,q)$ is an oscillating function of $q\in\R$ and therefore
has an increasing sequence of zeros, which we will denote by
\begin{equation*}
  q^{+}_{n,1}(\xi),\ q^{+}_{n,2}(\xi), \dots \ .  
\end{equation*}
A similar statement holds for $\Se_n(\xi,q)$ with zeros $q^{-}_{n,m}(\xi)$.

Let $\alpha_{n,m}\in\R$ denote the $m$-th positive zero of the
function $J_n$. As we saw in Section \ref{subsec:RQMfunctions},
$\Ce_n(\xi,q)$ and $\Se_n(\xi,q)$ are asymptotic to a nonvanishing
function times $J_n(\sqrt{q}\,\xi)$ as $\xi\to\infty$.  (In \cite[p.\
242]{McLachlan1951}, it is stated somewhat misleadingly that the
``asymptotic values'' of the zeros are $(n/2+m+3/4)\pi$, but this is
as $n\to\infty$, not for $\mu\to0$ when $n$ is fixed.) Consider the
specific value $\xi=\xi_\mu$ and note by \eqref{eq:defximu} that
$e^{-\xi_\mu}\sim2/\mu$ as $\mu\to0$. By Proposition
\ref{prop:psiphilimits} and the definition of $\alpha_{n,m}$, we have
the following characterization of $q^{\pm}_{n,m}(\xi_\mu)$ as
$\mu\Omega_\mu$ tends to the unit disk $B_1(0)$.
  
\begin{Lemma}\label{lemm:qlimit} For fixed $n\ge0$, $m > 0$, 
  $\displaystyle \frac{2}{\mu}\sqrt{q^{\pm}_{n,m}(\xi_\mu)}\to \alpha_{n,m}$
    as $\mu\to0$.
\end{Lemma}

Write $\zeta^\pm_{n,m}[\mu]=\zeta_n^\pm[q_{n,m}(\xi_\mu)]$. 

 \begin{figure}[t!]
  \pic{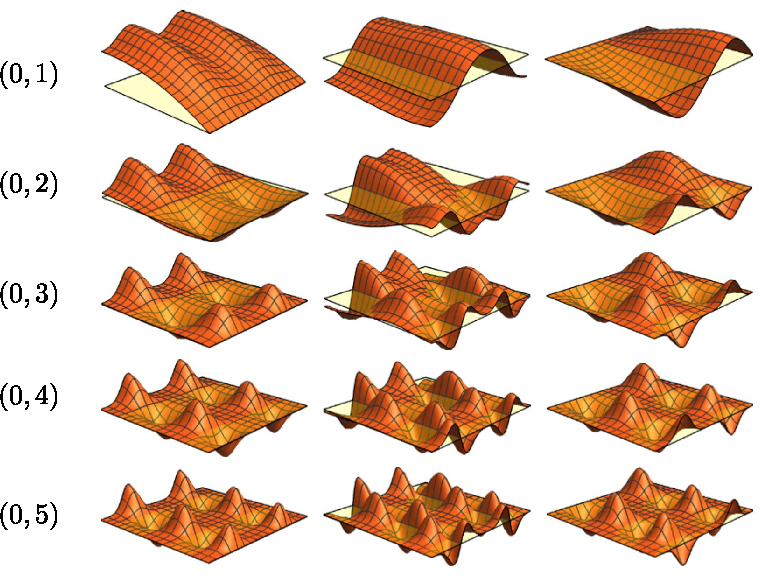}{.6,-13.7}{9}{scale=.48}
  \caption{ $Z^{+}_{0,m}[0.5]$ for low values of $m$.  In the left
    column it is observed that the scalar part always vanishes at the
    edge $\xi=\xi_\mu$. Vertical scales vary from image to
    image.\label{fig:mplus_a}}
\end{figure}

\begin{Proposition}[\cite{Gold1927,McLachlan1951}] \label{prop:scalarboundaryzero}
Let $0<\mu<1$. The doubly-indexed family of functions
\begin{equation*}  
\{\zeta^\pm_{n,m}[\mu]\colon n\ge0,\ m>0\}
\end{equation*}
enjoys the following completeness property. Let $F_0(x,y)=f_0(\xi,\eta)$
be any real-valued continuous function in the closed elliptical domain
$\overline{\Omega}_{\xi_\mu}$ such that $F_0(z)=0$
for every boundary point $z\in \partial \Omega_{\xi_\mu}$. Then $f_0$
may be expanded in the real Hilbert space $L^2(R_{\xi_{\mu}},\langle\ \rangle_{\xi_{\mu}})$  into a double
series in the functions $\zeta^\pm_{n,m}[\mu]$, which in fact form an
orthogonal basis of the closed Hilbert subspace that
they generate (i.e., the closure of the subspace of continuous functions with compact support).  Their norms squared are equal to
\begin{equation} \label{eq:norm}
\|\zeta^\pm_{n,m}[\mu]\|_{\xi_\mu}^2 = \int_0^{2\pi}\!\!\!\int_0^{\xi_{\mu}}
              \left(\zeta^\pm_{n,m}[\mu](\xi,\eta)\right)^2
              \frac{\mu^2}{2} (\cosh2\xi-\cos2\eta) \,d\xi d\eta.
\end{equation}
\end{Proposition}

The book \cite{McLachlan1951} corrects an error in the values of the
norms, which appears in the article \cite{McLachlan1945} by the same
author.
 
The RQM-functions whose scalar parts vanish on the boundary are now
constructed with the particular parameters $q^{\pm}_{n,m}$.
\begin{Definition} \label{Definition_zero-boundary_RQM-functions}
For $0<\mu<1$, the \textit{zero-boundary} RQM-functions $Z_{n,m}^\pm[\mu]$ are defined as follows:
  Let $n\ge0$, $m > 0$. Then
  \begin{equation} \label{RQM-functions}
  Z_{n,m}^\pm[\mu] =
    \Mat^\pm_{n}[\,2\sqrt{q^{\pm}_{n,m}(\xi_\mu)}\,].
  \end{equation}
  \end{Definition}

  Thus $\Sc Z_{n,m}^\pm[\mu] =\zeta^\pm_{n,m}[\mu]$. Although the
  $Z_{n,m}^\pm[\mu]$ are defined in all of $[0,\infty)\times[0,2\pi]$,
  their natural domain may be considered to be $R_{\xi_\mu}$.
  
 We illustrate some of the lowest RQM functions over two potential
periods and for a few $q$-values calculated with \textit{Mathematica}.
Each row in Figures \ref{fig:mplus_a} and \ref{fig:mplus_b} shows the
scalar, $\ei$, and $\ej$ components of a certain
$Z^{+}_{n,m}[\mu](\xi,\eta)$.
  
\begin{figure}[b!] 
   \pic{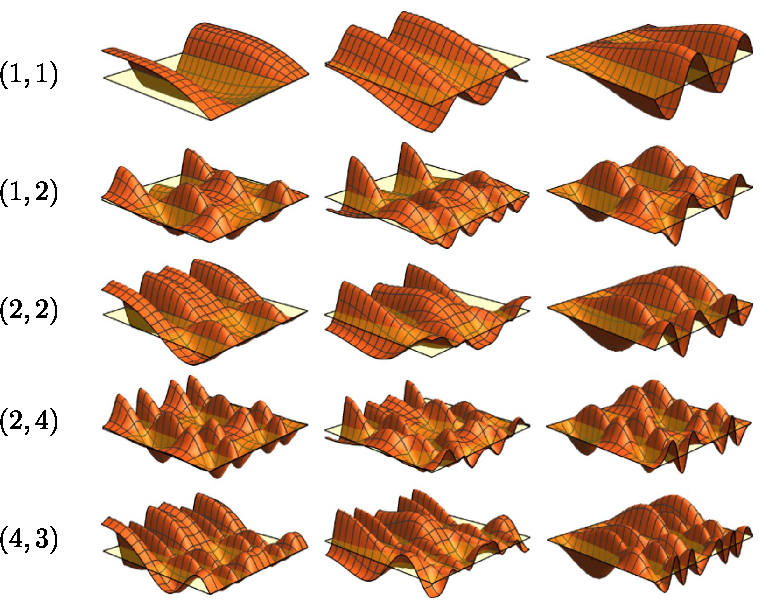}{.6,-14.5}{9.5}{scale=.48}
   \caption{ $Z^{+}_{n,m}[0.5]$ for selected  $(n,m)$ with $n>0$.
    \label{fig:mplus_b}}
\end{figure}

\begin{figure}[b!]
  \pic{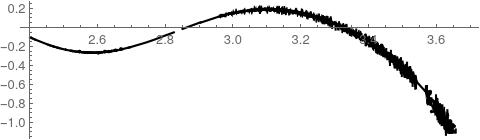}{2.2,-5.3}{3.3}{scale=.8}
  \caption{Values of $\Ce_2(\xi,q^+_{1,1}(\xi_\mu))$ calculated using
    machine precision arithmetic (with \textit{Mathematica}) for
    $\mu=0.1$. This illustrates the challenge of calculating of
    successive zeros $q^+_{n,m}$ as $m$
    increases.\label{fig:badgraph}}
\end{figure}
  
\subsection{Orthogonality and completeness}

In this part, we prove the orthogonality and completeness of the zero-boundary functions \eqref{RQM-functions}
over the rectangle \eqref{eq:Rectangle} with respect to the inner
product \eqref{eq:weighted} of quaternionic functions. We begin
with the orthogonality.

\begin{Theorem}[Orthogonality] \label{th:Zorthogonal} For fixed $\mu \in (0, 1)$, the collection
  of zero-boundary RQM-functions
\begin{equation} \label{eq:RQMset}
  \{ Z^+_{n,m}[\mu], \, Z^-_{n,m}[\mu]
  \colon  n=0,1,\dots,\  m=1,2,\dots \}
\end{equation}
is orthogonal over the rectangle $R_{\xi_\mu}$ in the
sense of the weighted $L^2$-inner product \eqref{eq:InnerProduct}, and their norms squared are equal to
\begin{align*}
  \|Z^{\pm}_{n,m}[\mu]\|^2_{\xi_\mu} =& \;
                               \|\zeta^\pm_{n,m}[\mu]\|_{\xi_\mu}^2+
                                        \frac{\mu^2}{4 q_{n,m}} \, (1+\delta_{0,n})\pi \left(\int_0^{\xi_\mu} \left(\psi^\pm_{n}[q_{n,m}]'(\xi)\right)^2 d\xi \right) \\
              &+ \frac{\mu^2}{4 q_{n,m}} \left(\int_0^{2\pi} \left(\phi^\pm_{n}[q_{n,m}]'(\eta)\right)^2 d\eta \right) \left(\int_0^{\xi_\mu} \left(\psi^\pm_{n}[q_{n,m}](\xi)\right)^2 d\xi \right),
\end{align*} 
where $\|\zeta^\pm_{n,m}[\mu]\|_{\xi_\mu}$ is given by \eqref{eq:norm}. Here $\delta_{0,n}$ is the  Kronecker symbol.
\end{Theorem}
\begin{proof}
  Suppose that $(n_1,m_1)\not=(n_2,m_2)$. We will treat the even and
  odd functions simultaneously. Since
  $\Sc(\overline{a}b)=\sum_0^2a_ib_i$ for $a,b\in\R^3$, by taking
  $a=Z^\pm_{n_1,m_1}[\mu](\xi,\eta)$ and
  $b=Z^\pm_{n_2,m_2}[\mu](\xi,\eta)$ (and choosing the $\pm$ signs
  independently) and referring to Proposition
  \ref{prop:mathieucomponents}, we observe first that
  \begin{align*}
    & \int_0^{2\pi} \!\! \int_{0}^{\xi_\mu}a_0b_0\,
      \frac{\mu^2}{2} (\cosh2\xi-\cos2\eta) \, d\xi d\eta =  0
  \end{align*}
  by Proposition \ref{prop:scalarboundaryzero}.

  For the remaining part of the inner product, a straightforward
  computation shows that
\begin{align*}
  a_1b_1 &= \frac{1}{\sqrt{q_1q_2}(\cosh2\xi-\cos2\eta)^2} \times\\
 & \ \big((\sinh^2\xi\cos^2\eta)\psi_1'\psi_2'\phi_1\phi_2-
  (\cosh\xi\sinh\xi\cos\eta\sin\eta)\psi_1\psi_2'\phi_1'\phi_2 \\
 &\   + (\cosh^2\xi\sin^2\eta)\psi_1\psi_2\phi_1'\phi_2' 
  -(\cosh\xi\sinh\xi\cos\eta\sin\eta)\psi_1'\psi_2\phi_1\phi_2'\big),
\end{align*}
where for simplicity we have suppressed the $\pm$ in the notation, and
have written $q_1$, $q_2$ in place of $q_{n_1,m_1}$,
$q_{n_2,m_2}$, etc. Similarly, $a_2b_2$ is an expression in
which the terms containing $\cosh\xi\sinh\xi\cos\eta\sin\eta$ have
the opposite sign as in $a_1b_1$ and cancel upon adding. Since
$\cosh2\xi-\cos2\eta = 2(\cosh^2\xi-\cos^2\eta)$, we arrive at the fact that
\begin{align*}
  \int_0^{2\pi}\!\!\!\int_0^{\xi_\mu}&  (a_1b_1+a_2b_2) \frac{\mu^2}{2} (\cosh2\xi-\cos2\eta) \,d\xi d\eta \\
   =\frac{\mu^{2}}{4\sqrt{q_1q_2}} &\int_0^{2\pi}\!\!\!\int_0^{\xi_\mu}(\psi_1'\psi_2'\phi_1\phi_2 + \psi_1\psi_2\phi_1'\phi_2') (\cosh2\xi-\cos2\eta) \, d\xi d\eta = 0
\end{align*}
by separating the integrals and applying the orthogonality relations
$\int_{0}^{2\pi} \phi_1\phi_2 d\eta = 0 = \int_{0}^{\xi_\mu} \psi_1\psi_2 d\xi$.
This shows the orthogonality of \eqref{eq:RQMset} and the norms are found by taking $(n_1,m_1)=(n_2,m_2)$.
\end{proof}

We now prove a completeness property of the set \eqref{eq:RQMset}.
This is relevant to modeling phenomena in an elliptical domain of
fixed ellipticity such as a lake, a uniform cable with elliptical
cross-section, modes in confocal annular elliptic resonators, etc.\
\cite{AlharganJudah1992,AlharganJudah1996,BaevaBaevKaplan1997,BrCoZa2021,HollandCable1992,Jeffreys1925,LewisDeshpande1979,Mathieu1868,Ruby1996}.

Consider the subcollection
\begin{align}
 \S_0(\xi_\mu) \subseteq\S(\xi_\mu)
\end{align}
of functions
$f\in\S(\xi_\mu)$ whose scalar part $f_0=\Sc f$ extends continuously
to the closed parametric rectangle $\overline{R}_{\xi_\mu}$ and
satisfies $f_0(\xi_\mu,\eta)=0$ for every $\eta\in[0,2\pi]$.

\begin{Theorem}[Completeness] \label{th:Zcomplete} Let
  $\lambda\in\R\setminus\{0\}$ and $0<\mu<1$. Suppose that
  $h \in \M_2(\xi_\mu;\lambda) \cap \S_0(\xi_\mu)$. Then there exist
  real constants $a^{\pm}_{n,m}$ such that
\[\lim_{N_1, \, N_2 \to \infty} \Bigl\| h(\xi,\eta) - \sum_{n=0}^{N_1} \sum_{m=1}^{N_2} a^{\pm}_{n,m} \, Z^{\pm}_{n,m}[\mu] \Bigr\|_{\xi_\mu} = 0.\]
\end{Theorem}

\begin{proof} Consider a square-integrable $f \in \S_0(\xi_\mu)$ which is $\lambda$-metamonogenic. For
  convenience, we write $\zeta^\pm_{n,m} =\zeta^\pm_{n,m}[\mu]$ and
  analogously $\phi^\pm_{n,m}$, $\psi^\pm_{n,m}$.  By Definition
  \ref{Definition_zero-boundary_RQM-functions} and Proposition
  \ref{prop:mathieucomponents},
\begin{align*}
 \langle  f, Z^{\pm}_{n,m}[\mu]  \rangle_{R_{\xi_\mu}}  
  &=  \, \frac{\mu^2}{2} \int_0^{2\pi}\!\!\!\int_0^{\xi_\mu}
    f_0(\xi,\eta) \, \zeta^{\pm}_{n,m}(\xi,\eta) (\cosh2\xi-\cos2\eta) \, d\xi \, d\eta \\
  &\quad + \frac{\mu^2}{2\lambda \sqrt{q_{n,m}}} \int_0^{2\pi}\!\!\!\int_0^{\xi_\mu}
    \bigg( \frac{\partial f_0}{\partial \xi}(\xi,\eta) \psi^{\pm \; \prime}_{n,m}(\xi) \phi^\pm_{n,m}(\eta)\\
   &\quad\hspace{20ex} + \frac{\partial f_0}{\partial \eta}(\xi,\eta) \psi^\pm_{n}(\xi) \phi^{\pm \; \prime}_{n,m}(\eta) \bigg) \, d\xi \, d\eta.
\end{align*}
Assume $f$ is orthogonal to every element of the set \eqref{eq:RQMset} in the sense of \eqref{eq:InnerProduct}. We apply integration by parts:
\begin{align*}
 &0 = \frac{\mu^2}{2} \int_0^{2\pi}\!\!\!\int_0^{\xi_\mu} f_0(\xi,\eta) \, \zeta^\pm_{n,m}(\xi,\eta) (\cosh2\xi-\cos2\eta) \, d\xi \, d\eta \\
&\hspace{-0.15cm} + \frac{\mu^2}{2\lambda \sqrt{q_{n,m}}} \left( \int_0^{2\pi} \Bigl( \left.\frac{\partial f_0}{\partial \xi} \, \psi^\pm_{n,m}(\xi) \right|^{\xi_\mu}_{0} 
- \int_0^{\xi_\mu} \psi^\pm_{n,m}(\xi) \, \frac{\partial^2 f_0}{\partial \xi^2}(\xi,\eta) \, d\xi \Bigr) \, \phi^\pm_{n,m}(\eta) \, d\eta \right. \\
&\qquad \; \qquad + \left. \int_0^{\xi_\mu} \Bigl( \left.\frac{\partial f_0}{\partial \eta} \, \phi^\pm_{n,m}(\eta) \right|^{2\pi}_{0} 
- \int_0^{2\pi} \phi^\pm_{n,m}(\eta) \, \frac{\partial^2 f_0}{\partial \eta^2}(\xi,\eta) \, d\eta \Bigr) \, \psi^\pm_{n,m}(\xi) \, d\xi \right).
\end{align*}
Since $\psi^-_{n,m}(0,\eta)=0 = \psi^\pm_{n,m}(\xi_\mu,\eta)$ and $(\partial f_0/\partial \xi)(0,\eta)$ is an odd function,
\[\int_{0}^{2\pi} \phi^+_{n,m}(\eta) \frac{\partial f_0}{\partial \xi}(0,\eta) \, d\eta = 0.\]
Therefore
\[\int_0^{2\pi} \left.\frac{\partial f_0}{\partial \xi}(\xi,\eta) \, \psi^{+}_{n,m}(\xi) \right|^{\xi_\mu}_{0} \phi^+_{n,m}(\eta) \, d\eta = 0. \]
The corresponding integral with $\psi^{-}_{n,m}(\xi)$, $\phi^{-}_{n,m}(\eta)$ is trivially zero. Also, from the periodicity in $\eta$ it follows that
\[\left.\frac{\partial f_0}{\partial \eta}(\xi,\eta) \, \phi^\pm_{n,m}(\eta) \right|^{2\pi}_{0} = 0, \]
so
\begin{align*}
 0 =& \, \frac{\mu^2}{2} \int_0^{2\pi}\!\!\!\int_0^{\xi_\mu} f_0(\xi,\eta) \, \zeta^\pm_{n,m}(\xi,\eta) (\cosh2\xi-\cos2\eta) \, d\xi \, d\eta \\
&- \frac{\mu^2}{2\lambda \sqrt{q_{n,m}}} \int_0^{2\pi}\!\!\!\int_0^{\xi_\mu} \left( \frac{\partial^2 f_0}{\partial \xi^2}(\xi,\eta) + \frac{\partial^2 f_0}{\partial \eta^2}(\xi,\eta)\right) \psi^\pm_{n,m}(\xi) \phi^\pm_{n,m}(\eta) \, d\xi \, d\eta.
\end{align*}

Since $f_0$ is $\lambda$-metaharmonic,
\begin{equation} \label{eq:zeroproduct}
\left(2\sqrt{q_{n,m}} + \lambda\right) \int_0^{2\pi}\!\!\!\int_0^{\xi_\mu} f_0(\xi,\eta) \, \zeta^\pm_{n,m}(\xi,\eta) (\cosh2\xi-\cos2\eta) \, d\xi \, d\eta = 0.
\end{equation}
It is instructive to observe how each of the properties characterizing
elements of $\S_0(\xi_\mu)$ is used in the steps in arriving at this
formula.

Now suppose first that $\lambda \neq -2 \sqrt{q_{n,m}}$ for all $n,m$
(this automatically holds when $\lambda>0$).  Thus the integral of
\eqref{eq:zeroproduct} vanishes for all $(n,m)$, and then by
Proposition \ref{prop:scalarboundaryzero}, it follows that $f_0 = 0$.
Then by Lemma \ref{lemm:zeroscalar}, in fact, $f = 0$ identically.

On the other hand, assume that $\lambda = -2\sqrt{q_{n_0,m_0}}$
for at least one pair of indices $(n_0,m_0)$. In this situation,
the fact that \eqref{eq:zeroproduct} holds for all $(n,m)$ only implies that
$f_0 \in \mbox{Span}\{\zeta^{\pm}_{n,m}[\mu]: q_{n,
  m}=4\lambda^2\}$.  There is at most one $m$ for any given $n$
such that $q_{n, m}$ has a given value, so we can write the functions
under consideration as $\zeta^{\pm}_{n,m(n)}[\mu]$.  In other words,
we may write the scalar part as
\begin{equation*}
f_0 = \sum_n a^{\pm}_n \, \zeta^{\pm}_{n,m(n)}[\mu],
\end{equation*}
where the sum is over those values of $n$ for which there is an $m(n)$
satisfying $q_{n, m(n)}=4\lambda^2$. Define $g$ by
\eqref{Function_construction}. Following the reasoning of the proof of
Theorem \ref{Theorem_completeness}, we see that
$f=g=\sum a^\pm_nZ^{\pm}_{n,m(n)}[\mu]$, where in fact
$Z^{\pm}_{n,m(n)}[\mu]=M^\pm_n[\lambda]$. But as $f$ is orthogonal to
all $Z_{n,m}^\pm[\mu]$, we have again that $f=0$.

Recall that by construction, $Z_{n,m}^\pm[\mu]\in\S_0(\xi_\mu)$. We
have shown that any $f\in\S_0(\xi_\mu)\cap\M_2(\xi_\mu,\lambda)$, which
is orthogonal to all of the $Z_{n,m}^\pm[\mu]$, must vanish
identically. Thus
$\S_0(\xi_\mu)\cap\M_2(\xi_\mu,\lambda)\subseteq\mbox{Span}\{Z_{n,m}^\pm[\mu]\}$;
i.e., we can expand
$h= \sum_{n=0}^\infty\sum_{m=1}^\infty a^{\pm}_{n,m} \,
Z^{\pm}_{n,m}[\mu]$ as required.
\end{proof} 

The apparent asymmetry of positive and negative $\lambda$ in the proof of Theorem \ref{th:Zcomplete} is resolved by noting that the
conjugate functions $\overline{Z^\pm_{n,m}[\mu]}$ satisfy the same completeness property.

\medskip

\noindent\textit{Degeneration of zero-boundary RQMs as the ellipse
  tends to a disk.} As a consequence of Theorem 
\ref{th:Zcomplete}, all $\lambda$-metamonogenic functions in the
ellipse $\Omega_\mu$ with scalar part tending to zero at the boundary
can be expanded in terms of the expressions of the $Z^\pm_{n,m}[\mu]$
in Cartesian coordinates.  Since
$2\sqrt{q^\pm_{n,m}(\mu)}/\mu \to \alpha_{n,m}$ as $\mu\to0$, by
Definition \ref{Definition_zero-boundary_RQM-functions} and Theorem
\ref{theorem:mathieulimits}, we have the $\alpha_{n,m}$-metamonogenic
limit functions
\begin{equation} \label{eq:RQM_limiting_case}
   Z^{\pm}_{n,m}[0](w) = \lim_{\mu\to0}q_{n,m}[\mu]^{n/2}Z^{\pm}_{n,m}[\mu](\xi,\eta).
\end{equation}

\begin{Corollary}
  The set $\{Z^{\pm}_{n,m}[0]\colon n\ge0,\ m\ge1\}$ is an orthogonal
  Hilbert basis for the space of continuous functions on $B_1(0)$ with
  scalar part vanishing on the boundary circle, and their norms
  squared are given by
\begin{equation*}
\| Z^{\pm}_{n,m}[0] \|^2_{B_1(0)} = \beta^2_{n,m} J_{n+1}(\alpha_{n,m}) (1+\delta_{0,n})\pi.
\end{equation*}
\end{Corollary}
\begin{proof}
  Suppose that $(n_1,m_1) \neq (n_2,m_2)$. By Theorem
  \ref{theorem:mathieulimits} and equation
  \eqref{eq:RQM_limiting_case}, direct computations show that
\begin{align*}
&\langle Z^{\pm}_{n_1,m_1}[0], Z^{\pm}_{n_2,m_2}[0] \rangle_{B_1(0)} \\
=& \, \beta_{n_1,m_1} \beta_{n_2,m_2} \int_{0}^{2\pi} \phi_{n_1}^{\pm}(\eta) \phi_{n_2}^{\pm}(\eta) \, d\eta \\
&\times \int_{0}^{1} \Bigl(J_{n_1}(\alpha_{n_1,m_1}\rho) J_{n_2}(\alpha_{n_2,m_2}\rho) + J^{\prime}_{n_1}(\alpha_{n_1,m_1}\rho) J^{\prime}_{n_2}(\alpha_{n_2,m_2}\rho) \Bigr) \rho \, d\rho \\
&+ \frac{n_1 n_2 \beta_{n_1,m_1} \beta_{n_2,m_2}}{\alpha_{n_1,m_1} \alpha_{n_2,m_2}} \int_{0}^{2\pi} \phi_{n_1}^{\mp}(\eta) \phi_{n_2}^{\mp}(\eta) \, d\eta 
\int_{0}^{1} J_{n_1}(\alpha_{n_1,m_1}\rho) J_{n_2}(\alpha_{n_2,m_2}\rho) \frac{1}{\rho} \, d\rho \\
=& \, 0
\end{align*}
when $n_1 \neq n_2$.

Now, let $n_1=n_2=n$. Using the well known recurrence relations for the Bessel functions
\begin{align*}
2J^{\prime}_{n}(u) &= J_{n+1}(u) - J_{n+1}(u), \\
\frac{2n}{u} J_{n}(u) &= J_{n+1}(u) + J_{n+1}(u),
\end{align*}
and the orthogonality relation
\begin{align*}
&(\alpha^2_{n,m_2} - \alpha^2_{n,m_1}) \int J_{n}(\alpha_{n,m_1}\rho) J_{n}(\alpha_{n,m_2}\rho) \rho \, d\rho \\
&= \rho \Bigl( \alpha_{n,m_1} J_{n-1}(\alpha_{n,m_1}\rho) J_{n}(\alpha_{n,m_2}\rho) - \alpha_{n,m_2} J_{n}(\alpha_{n,m_1}\rho) J_{n-1}(\alpha_{n,m_2}\rho) \Bigr)
\end{align*}
we find
\begin{align*}
&\langle Z^{\pm}_{n,m_1}[0], Z^{\pm}_{n,m_2}[0] \rangle_{B_1(0)} \\
=& \, \beta_{n,m_1} \beta_{n,m_2} (1+\delta_{0,n}) \pi \Bigl( \frac{1}{2} \int_{0}^{1} J_{n-1}(\alpha_{n,m_1}\rho) J_{n-1}(\alpha_{n,m_2}\rho) \rho \, d\rho \\
&+ \int_{0}^{1} J_{n}(\alpha_{n,m_1}\rho) J_{n}(\alpha_{n,m_2}\rho) \rho \, d\rho + \frac{1}{2} \int_{0}^{1} J_{n+1}(\alpha_{n,m_1}\rho) J_{n+1}(\alpha_{n,m_2}\rho) \rho \, d\rho \Bigr) \\
=& \, 0
\end{align*}
for $m_1 \neq m_2$. The norms are obtaining by equating the indexes.
\end{proof}


\section{Time-dependent metaharmonic and metamonogenic functions\label{sec:application}}

The $\lambda$-RQM functions relate to some partial differential
equations of mathematical physics. We will examine perhaps the
simplest possible example here. Consider the imaginary-time wave
equation $(\Delta + K^2\frac{\partial^2}{\partial t^2})V=0$ for an
${\R^3}$-valued function $V(x,y,t)$ in an elliptical domain
$\Omega_{\xi_0}$. This differs from the standard wave equation by
replacing $K^2$ with $-K^2$, i.e.\ applying the Wick transformation
$t=i\tau$ to the time variable
\cite{Burgess2003}.
Converting the spatial variables to elliptical
coordinates, we have the equivalent form
\begin{align} \label{eq:helmholtzK}
  \big(\L + K^2\frac{\partial^2}{\partial t^2} \big) v = 0,
\end{align}
where $\L$ is defined by \eqref{eq:defLmu}.  

The time-dependent solution $v = v(\xi,\eta,t)$ of
\eqref{eq:helmholtzK} should satisfy at least
$v(\cdot,\cdot,t)\in\S(\xi_0)$ for every $t\ge0$. The point of
interest is to examine the natural reduced-quaternionic extensions of
the real-valued solutions. One may note the analogy to the process of
holomorphic extension of the classical trigonometric and related
functions from the real to the complex domain. Since the operator
$\L + K^2 (\partial^2/\partial t^2)$ has only real ingredients, it
applies to $v=v_0 + \ei v_1 + \ej v_2$ by operating independently on
each component.

\def\uv{\underline{v}}

We consider the boundary value problem for solutions with scalar part
vanishing at the boundary $\Omega_{\xi_0}$. Thus we assume an
$\R$-valued initial condition $\uv(\xi,\eta)\in\S_0(\xi_0)$. One finds
coefficients $a^\pm_{n,m}\in\R$ according to Proposition
\ref{prop:scalarboundaryzero}, representing this continuous function as
\begin{align}  \label{eq:vexpansion}
  \uv = \sum_{n=0}^\infty\sum_{m=1}^\infty a^\pm_{n,m} \zeta^\pm_{n,m} . 
\end{align}
Introduce the constants
\begin{align}  \label{eq:defomega}
  \omega_{n,m} = \frac{2\sqrt{q_{n,m}(\mu_0)}}{K},
\end{align}
where the functions $q_{n,m}$ given in Section
\ref{Zero-boundary_RQM-functions} are applied to $\mu_0=1/\cosh\xi_0$,
cf.\ \eqref{eq:defximu}.  

\begin{Theorem} \label{theo:wavefunction} The function $v$ given by
\begin{equation} \label{eq:totalwavefunction}
v(\xi,\eta,t) =
 \sum_{n=0}^{\infty} \sum_{m=1}^{\infty} a^\pm_{n,m}\,
  Z^{\pm}_{n,m}[\mu_0](\xi,\eta)\, e^{\omega_{n,m} t}         
\end{equation}   is a reduced-quaternionic solution of the imaginary-time wave equation \eqref{eq:helmholtzK} in $\S_0(\xi_0)$, where
 \[a^\pm_{n,m} = \frac{\displaystyle \Sc \int_{0}^{2\pi} \!\!
     \int_{0}^{\xi_0} \overline{Z^{\pm}_{n,m}[\mu_0]}(\xi,\eta)
     \uv(\xi,\eta) \mu^2_0 (\cosh^2 \xi - \cos^2\eta) \,d\xi d\eta}{\|
     Z^{\pm}_{n,m}[\mu_0]\|^2_{\xi_0}}.\] In fact, $v$ is the only
 real-analytic solution of \eqref{eq:helmholtzK} whose
 scalar part $v_0$ at time $t=0$ is equal to $\uv$ and
 \begin{align} \label{eq:vtime0}
   \left.\frac{\partial}{\partial t} \right|_{t=0} v_0 =
   \sum_{n=0}^\infty \sum_{m=1}^\infty a^\pm_{n,m} \, \omega^\pm_{n,m} \, \zeta^\pm_{n,m}[\mu_0].
 \end{align}
  Further, $v$ is a solution of the
 time-dependent Moisil-Teodorescu equation
  \begin{align}    \label{eq:timedepmetamonog}
     \big(\D + K\frac{\partial}{\partial t} \big )v=0 .
  \end{align}
 \end{Theorem}

\begin{figure}[p!]
  \pic{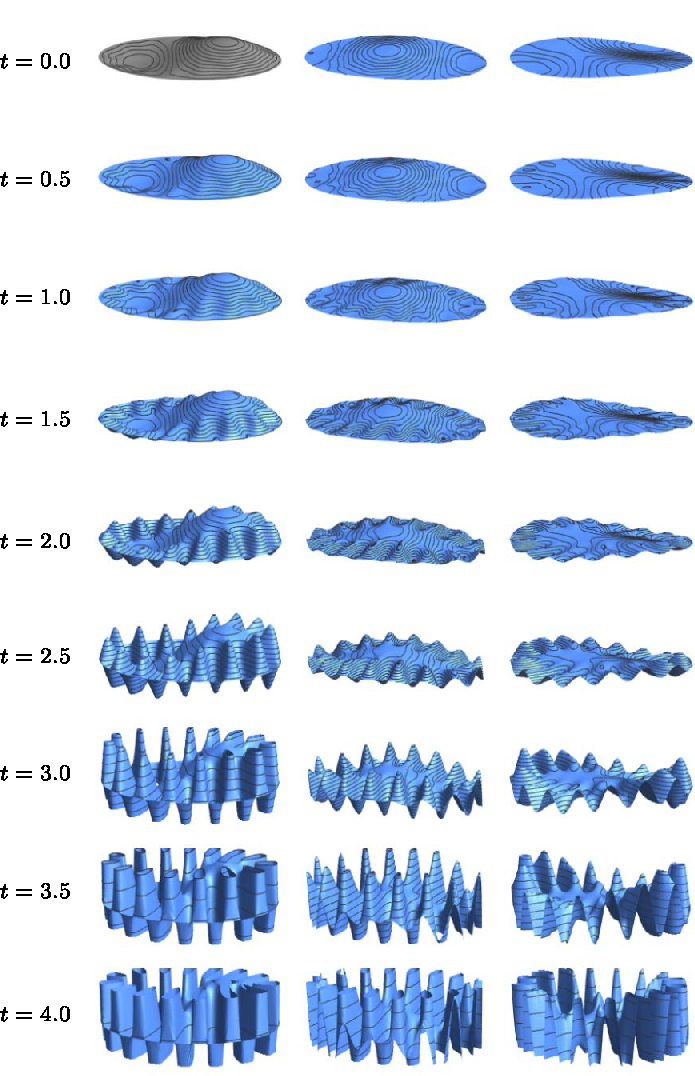}{.2,-24.4}{17}{scale=.505}
  \caption{Time evolution of wave function \eqref{eq:vexpansion} in an ellipse of eccentricity $\mu_0=0.7$.  Here $v$ is defined by coefficients
$a^+_{n,0}=a^+_{n,1}=a^-_{n,1}=1$, $a^+_{n,15}=a^-_{n,15}=0.0001$, and
all other $a^\pm_{n,m}=0$ and wave parameter $K=10$.  The initial condition $\uv$ created by these parameters is shown
in the upper left of the figure.
\label{fig:wave}}
\end{figure}

\begin{proof}
  First we show \eqref{eq:timedepmetamonog}. In analogy to \eqref{eq:Factorization},
  \begin{align*} 
    \L + K^2\frac{\partial^2}{\partial t^2} =
    -\big(\D-K\frac{\partial}{\partial t}\big)
     \big(\D+K\frac{\partial}{\partial t}\big).
  \end{align*}
But 
\begin{align*}
  \big(\D+K\frac{\partial}{\partial t}\big)
  Z^{\pm}_{n,m}[\mu_0](\xi,\eta)\, e^{\omega_{n,m} t} &=
  (\D + K\omega_{n,m}) Z^{\pm}_{n,m}[\mu_0] e^{\omega_{n,m} t} =0
\end{align*}
because $Z^{\pm}_{n,m}[\mu_0]$ is $K\omega_{n,m}$-metamonogenic by
\eqref{eq:defomega}. Substitute this into \eqref{eq:totalwavefunction}
to obtain \eqref{eq:timedepmetamonog}, in the light of Proposition
\ref{prop:convergence}. It follows automatically that $v$ satisfies \eqref{eq:helmholtzK} as well.

By construction, $v_0(\xi,\eta,0)=\uv(\xi,\eta)$, and it is easy to
verify \eqref{eq:vtime0}. Suppose $u(\xi,\eta,t)$ is another solution
sharing the property $u_0(\xi,\eta,0)=\uv(\xi,\eta)$ and with
$(\partial u_0/\partial t)_{t=0}$  equal to the right-hand side of
\eqref{eq:vtime0}.  Consider the difference $w=v-u$, whose scalar part
$w_0(\xi,\eta,t)$ vanishes identically at time $t=0$. The
reduced-quaternionic equation
$(\L + K^2\,\partial^2/\partial t^2)w = 0$ tells us that in particular
\[ (\L + K^2\frac{\partial^2}{\partial t^2})w_0 = 0 ,\] which together
with the vanishing initial conditions and the fact that $w_0$ is
real-analytic implies by the Cauchy-Kovalevskaya theorem
\cite{Petrovsky1954} that $w_0(\xi,\eta,t)=0$ for all $t>0$ and
all $(\xi,\eta)\in R_{\xi_0}$.

Recalling \eqref{eq:timedepmetamonog}, we see that
$\D w = -K\,\partial w/\partial t$.  Since $w_0=0$, the left-hand side
of this equality lies in the subspace $\R\oplus\ek\R$  of $\H$  while
the right-hand side is in $\ei\R\oplus\ej\R$. It follows that $\D w=0$
with $w$ independent of $t$. By Lemma \ref{lemm:zeroscalar}, we have
$w=0$, which yields the uniqueness of $v$.
\end{proof}

\begin{table}[!t]
\begin{center}
 \begin{tabular}{r |r@{.}l r@{.}l|r@{.}l r@{.}l }
    $m$ & \multicolumn{2}{c}{$q^+_{0m}$}  &\multicolumn{2}{c|}{$q^-_{0m}$} &
  \multicolumn{2}{c}{$\omega^+_{0m}$} &\multicolumn{2}{c}{$\omega^-_{0m}$}\\ \hline
  1 & 2&21929 & 3&08131 & 0&297946 & 0&351073 \\
  2 & 3&91836 & 4&7426 & 0&395897 & 0&43555 \\
  3 & 6&14002 & 6&85948 & 0&495581 & 0&523812 \\
  4 & 8&86358 & 9&43514 & 0&595435 & 0&614333 \\
  5 & 12&0539 & 12&4629 & 0&694373 & 0&706057 \\
  6 & 15&6683 & 15&9304 & 0&791664 & 0&798259 \\
  7 & 19&6714 & 19&8229 & 0&887049 & 0&890458 \\
  8 & 24&0454 & 24&1259 & 0&980723 & 0&982362 \\
  9 & 28&787 & 28&827 & 1&07307 & 1&07382 \\
  10 & 33&8979 & 33&9168 & 1&16444 & 1&16476 \\
  11 & 39&3799 & 39&3885 & 1&25507 & 1&25521 \\
  12 & 45&2335 & 45&2373 & 1&34512 & 1&34517 \\
  13 & 51&4582 & 51&4599 & 1&43469 & 1&43471 \\
  14 & 58&053 & 58&0537 & 1&52385 & 1&52386 \\
  15 & 65&0165 & 65&0168 & 1&61266 & 1&61266  
  \end{tabular}
\end{center}
\caption{Values of parameters $q^\pm_{mn}$, $\omega^\pm_{mn}$ used in example of
  a function \eqref{eq:vexpansion}. Only $m=0$, $1$, and $15$ appear with nonzero coefficients $a^\pm_{mn}$. This permits appreciating the effect of the larger
  exponents $\omega^\pm_{15,1}$ which give dominating terms for larger values
  of $t$ as shown in the lower rows of Figure \ref{fig:wave}.
  \label{tab:wave}}
\end{table}

One may use $\uv$ to generate a function in $R_{\xi_1}$ for arbitrary
$\xi_1>0$ simply by rescaling the $\xi$-variable, that is,
$\uv[\mu_1](\xi,\eta)=\uv[\mu_1]((\xi_1/\xi_0)\xi,\eta)$ where
$\mu_1=1/\cosh\xi_1$.  The solutions $v[\mu]$ of \eqref{eq:helmholtzK}
given by Theorem \ref{theo:wavefunction} form a one-parameter family
of functions whose limit as $\mu\to0$ may be obtained via Theorem
\ref{theorem:mathieulimits}.

An example of a wave function \eqref{eq:vexpansion} is given in Figure
\ref{fig:wave}, and the relevant coefficients $q^\pm_{mn}$,
$\omega^\pm_{mn}$ are listed in Table \ref{tab:wave}.


\section{Concluding remarks}

We have introduced a generalization of the classical
$\lambda$-metaharmonic Mathieu functions \cite{Mathieu1868} within the
reduced-quaternionic analysis setting, called $\lambda$-RQM
functions. The $\lambda$-RQM functions are orthogonal to one another
with respect to the $L^2$-norm over confocal ellipses for the same
Mathieu parameter $q$ and form a complete system in the Hilbert space
of square-integrable $\lambda$-metamonogenic functions. Further, we
have studied the zero-boundary $\lambda$-RQM-functions, which are
$\lambda$-RQM functions whose scalar part vanishes on the boundary of
the ellipse. The limiting values of the $\lambda$-RQM functions as the
eccentricity of the ellipse tends to zero are expressed in terms of
Bessel functions of the first kind and form a complete orthogonal
system for $\lambda$-metamonogenic functions with respect to the
$L^2$-norm on the unit disk. The fundamental properties of the
zero-boundary $\lambda$-RQM functions and their connections to the
time-dependent solutions of the imaginary-time wave equation in the
elliptical coordinate system were also discussed. Examples of the
$\lambda$-RQM functions were calculated using \textit{Mathematica}.
 
When the idea of $\lambda$-RQM functions is applied to the ordinary
(real-time) wave equation obtained by using $-K^2$ in place of $K^2$,
one finds that the complex-valued operator $D+i(\partial/\partial t)$
appears. While the functions we have studied here have been expressed
for simplicity in the context of real quaternions, all results extend
immediately to reduced complex quaternions.

It is of interest to explore other properties of the imaginary-time
wave function \eqref{eq:totalwavefunction} and its transition to
circular membranes in more detail. We will extend our results in a
future paper to Mathieu functions taking values in the space of full
quaternions (with a right $\H$-linear structure), using the
theoretical basis presented in \cite{LuPeRoSha2013,Morais2014}.
 
The $\lambda$-RQM functions give an elegant generalization of the
classical Mathieu functions inspired by the factorization of the
imaginary-time wave operator \eqref{eq:helmholtzK}. We hope that this
knowledge will enrich the understanding of such functions and
Moisil-Teodorescu operators in higher dimensions and serve as tools in
the application of quaternionic analysis to mathematical physics and
related fields.


\newcommand{\authors}[1]{\textsc{#1}}
\newcommand{\booktitle}[1]{\textit{#1}}
\newcommand{\articletitle}[1]{``#1''}
\newcommand{\journalname}[1]{\textit{#1}}
\newcommand{\volnum}[1]{\textbf{#1}}

\end{document}